\documentclass[12pt,a4paper]{article}
\usepackage{mathrsfs}
\usepackage{epsfig, graphicx}
\usepackage{latexsym,amsfonts,amsbsy,amssymb}
\usepackage{amsmath,amsthm}
\usepackage{color}
\usepackage[colorlinks, citecolor=blue]{hyperref}
\makeatletter 

\@addtoreset{equation}{section}
\makeatother 
\allowdisplaybreaks 

\newtheorem{theorem}{Theorem}[section]

\newtheorem{remark}{Remark}[section]
\newtheorem{algorithm}{Algorithm}[section]

\newtheorem{example}{Example}[section]

\begin{document}
\title{A Full Multigrid Method for Nonlinear Eigenvalue Problems\footnote{This work was supported
 in part by National Science Foundations of China
(NSFC 91330202, 11371026, 11001259, 11031006, 2011CB309703)
and the National Center for Mathematics and Interdisciplinary Science,
CAS and the President Foundation of AMSS-CAS.}}
\author{
Shanghui Jia\footnote{School of Applied
Mathematics, Central University of Finance and Economics,
  Beijing 100081, China(shjia@lsec.cc.ac.cn)},\ \
Hehu Xie\footnote{LSEC, ICMSEC,Academy of Mathematics and Systems Science, Chinese Academy of
Sciences, Beijing 100190, China (hhxie@lsec.cc.ac.cn)},\ \
Manting Xie\footnote{LSEC, ICMSEC,Academy of Mathematics and Systems Science, Chinese Academy of
Sciences, Beijing 100190, China (xiemanting@lsec.cc.ac.cn)} \ \ and \ \
Fei Xu\footnote{LSEC, ICMSEC,Academy of Mathematics and Systems Science, Chinese Academy of
Sciences, Beijing 100190, China (xufei@lsec.cc.ac.cn)}
}
\date{}
\maketitle
\begin{abstract}
This paper is to introduce a type of full multigrid method for the nonlinear eigenvalue problem.
The main idea is to transform the solution of nonlinear eigenvalue problem into a series of
solutions of the corresponding linear boundary value problems on the sequence of finite element
spaces and nonlinear eigenvalue problems on the coarsest finite element space.
The linearized boundary value problems are solved by some multigrid iterations. Besides the
multigrid iteration, all other efficient iteration methods for solving boundary value
problems can serve as the linear problem solver. We will prove that the computational work
of this new scheme is truly optimal, the same as solving the linear corresponding boundary
value problem. In this case, this type of iteration scheme certainly improves the overfull
efficiency of solving nonlinear eigenvalue problems.
Some numerical experiments are presented to validate the efficiency of the new method.


\end{abstract}

\section{Introduction}
In recent years, much effort has been devoted to the study of
problems in solving large scale eigenvalue problems. Among these eigenvalue problems, there exist many
nonlinear eigenvalue problems
\cite{Bao,BaoDu,CancesChakirMaday,ChenGongHeYangZhou,ChenGongZhou,ChenHeZhou,KohnSham,Martin,ParrYang,SulemSulem},
for instance the calculation of the Gross-Pitaevskii equation describing the ground states of
Bose-Einstein condensates
\cite{Bao,BaoDu}  or the Hartree-Fock and Kohn-Sham equations used to calculate ground state
electronic structures of molecular systems \cite{ChenGongHeYangZhou,ChenGongZhou,KohnSham,Martin,ParrYang,SulemSulem}
from physics, chemistry and material science. However, these high-dimensional eigenvalue problems
 are always very difficult to solve.

The multigrid and multilevel
methods \cite{BankDupont,Bramble,BramblePasciak,BrambleZhang,BrandtMcCormickRuge,Hackbusch_Book,
McCormick,ScottZhang,Shaidurov,ToselliWidlund,Xu}
provide optimal order algorithms for solving boundary value problems.
The error bounds of the approximate solutions obtained from these efficient numerical algorithms are comparable
to the theoretical bounds determined by the finite element discretization. But there is no many efficient numerical
methods for solving nonlinear eigenvalue problems
with optimal complexity. Recently, a type of multigrid method for eigenvalue problems has been proposed
in \cite{LinXie,Xie_IMA,Xie_Nonconforming,Xie_JCP,XieXie}. The aim of this paper
is to present a full multigrid method (sometimes also referred as nested finite element method)
for solving nonlinear eigenvalue problems based on the combination of the multilevel correction
method \cite{Xie_IMA,Xie_JCP} and the multigrid iteration for
boundary value problems. Comparing with the method in \cite{LinXie,Xie_IMA,Xie_JCP,XieXie},
the difference is that it is not necessary to solve the linear boundary value problem
exactly in each correction step. We only
get an approximate solution with some multigrid iteration steps. In this new version of
multigrid method, solving nonlinear eigenvalue problem will not be much more difficult than the multigrid scheme for the
corresponding linear boundary value problems.

An outline of the paper goes as follows. In Section 2, we introduce the
finite element method for eigenvalue problem and state some basic assumptions
about the error estimates. A type of full multigrid algorithm for solving the nonlinear
eigenvalue problem and the corresponding computational work estimate are given in Section 3.
Two numerical examples are presented in section 4 to validate our theoretical analysis.
Some concluding remarks are given in the last section.

\section{Finite element method for nonlinear eigenvalue problem}\label{sec;preliminary}
This section is devoted to introducing some notation and the finite element
method for nonlinear eigenvalue problem. In this paper, the standard notation
for Sobolev spaces $W^{s,p}(\Omega)$ and their
associated norms and semi-norms (cf. \cite{Adams}) will be used. For $p=2$, we denote
$H^s(\Omega)=W^{s,2}(\Omega)$ and
$H_0^1(\Omega)=\{v\in H^1(\Omega):\ v|_{\partial\Omega}=0\}$,
where $v|_{\Omega}=0$ is in the sense of trace,
$\|\cdot\|_{s,\Omega}=\|\cdot\|_{s,2,\Omega}$.  Let $V=H_0^1(\Omega)$
and $\|\cdot\|_s$ denote $\|\cdot\|_{s,\Omega}$ for simplicity.
To facilitate the following instructions, the letter $C$ (with or without subscripts) denotes a generic
positive constant which may be different at its different occurrences through the paper.

This paper is concerned with the following nonlinear elliptic eigenvalue problem:
Find $(\lambda, u)\in \mathcal{R}\times H_0^1(\Omega)$ such that
\begin{equation}\label{nonlinear_eigen_problem}
\left\{
\begin{array}{rcl}
-\nabla\cdot(\mathcal{A}\nabla u)+f(x,u)&=&\lambda u, \quad {\rm in} \  \Omega,\\
u&=&0, \ \  \quad {\rm on}\  \partial\Omega,\\
\int_{\Omega}u^2d\Omega&=&1,
\end{array}
\right.
\end{equation}
where $\mathcal{A}$ is a symmetric and positive definite matrix with suitable
regularity, $f(x,u)$ is a nonlinear function corresponding to the variable $u$,
$\Omega\subset\mathcal{R}^d$ $(d=2,3)$ is a bounded domain with Lipschitz
boundary $\partial\Omega$.

In order to use the finite element method for
the eigenvalue problem (\ref{nonlinear_eigen_problem}), we define
the corresponding variational form as follows:
Find $(\lambda, u )\in \mathcal{R}\times V$ such that $b(u,u)=1$ and
\begin{eqnarray}\label{weak_form_nonlinear_pde}
a(u,v)&=&\lambda b(u,v),\quad \forall v\in V,
\end{eqnarray}
where
\begin{equation*}\label{def_a_b}
a(u,v) = \int_{\Omega}\Big(\mathcal{A}\nabla u\cdot\nabla v+f(x,u)v\Big)d\Omega,\quad
 b(u,v) = \int_{\Omega}uv d\Omega.
\end{equation*}
For simplicity of describing and understanding, we only consider
the numerical method for the simple eigenvalue case.

Now, let us define the finite element approximations for the problem
(\ref{weak_form_nonlinear_pde}). First we generate a shape-regular
decomposition of the computing domain $\Omega\subset \mathcal{R}^d\
(d=2,3)$ into triangles or rectangles for $d=2$ (tetrahedrons or
hexahedrons for $d=3$) (cf. \cite{BrennerScott,Ciarlet}).
The diameter of a cell $K\in\mathcal{T}_h$ is denoted by $h_K$ and
the mesh size $h$ describes  the maximum diameter of all cells
$K\in\mathcal{T}_h$. Based on the mesh $\mathcal{T}_h$, we can
construct a finite element space denoted by $V_h \subset V$.
For simplicity, we set $V_h$ as the linear finite element space
which is defined as follows
\begin{equation}\label{linear_fe_space}
V_h = \big\{ v_h \in C(\Omega)\ \big|\ v_h|_{K} \in \mathcal{P}_1,
\ \ \forall K \in \mathcal{T}_h\big\},
\end{equation}
where $\mathcal{P}_1$ denotes the linear function space.

The standard finite element scheme for eigenvalue
 problem (\ref{weak_form_nonlinear_pde}) is:
Find $(\bar{\lambda}_h, \bar{u}_h)\in \mathcal{R}\times V_h$
such that $b(\bar{u}_h,\bar{u}_h)=1$ and
\begin{eqnarray}\label{fem_nonlinear_pde}
a(\bar{u}_h,v_h)
&=&\bar{\lambda}_h b(\bar{u}_h,v_h),\quad\ \  \ \forall v_h\in V_h.
\end{eqnarray}

Define a bilinear form $\widehat{a}(\cdot,\cdot)$ as follows
$$\widehat{a}(w,v) = \int_\Omega\mathcal{A}\nabla w\cdot\nabla vd\Omega,\ \ \ \forall w\in V,\ \ \forall v\in V$$
and the correspoding norm $\|\cdot\|_a$ is defined by
\begin{eqnarray}\label{def of a norm}
\|v\|_a=\sqrt{\widehat{a}(v,v)}.
\end{eqnarray}
Denote
\begin{eqnarray}
\delta_h(u)=\inf_{v_h\in
V_h}\|u-v_h\|_{a}.
\end{eqnarray}

For designing and analyzing the full multigrid method, we state the
following assumption for the nonlinear function $f(x,\cdot): \mathcal{R}\times V\rightarrow V$.

\textbf{Assumption A}: The nonlinear function $f(x,\cdot)$ has the following estimate
\begin{eqnarray}\label{Err_functional_Global_0_Norm}
|(f(x,w)-f(x,v),\psi)|\leq C_f\|w-v\|_0\|v\|_a,\ \ \ \forall w\in V,\ \forall v\in V,\ \forall\psi\in V.
\end{eqnarray}

For generality, we only state the following assumptions about the error estimate
for the eigenpair approximation $(\bar{\lambda}_h, \bar{u}_h)$ defined by (\ref{fem_nonlinear_pde})
 (see, e.g., \cite{CancesChakirMaday,ChenGongZhou} for practical examples).

\textbf{Assumption B1}:\
The eigenpair approximation $(\bar\lambda_h,\bar u_h)$ of (\ref{fem_nonlinear_pde}) has the following
error estimates
\begin{eqnarray}
\|u-\bar u_h\|_a &\leq&(1+C_u\eta_a(V_h))\delta_h(u),\label{Err_Eigenfunction_Global_1_Norm} \\
|\lambda-\bar\lambda_h|+\|u-\bar u_h\|_0&\leq & C_u\eta_a(V_h)\|u-\bar u_h\|_a,\label{Err_Eigenfunction_Global_0_Norm}
\end{eqnarray}
where $\eta_a(V_h)$ depends on the finite dimensional space $V_h$ and has the following property
\begin{eqnarray}\label{relationship_eta_space}
\lim_{h\rightarrow 0}\eta_a(V_h)=0, \ \ \ \eta_a(\widetilde{V}_{h})\leq \eta_a(V_h)\ \ {\rm if}\
V_h\subset \widetilde{V}_h\subset V.
\end{eqnarray}
Here and hereafter $C_u$ is some constant depending on regularity of mesh and the exact eigenfunction
but independent of the mesh size $h$.

\textbf{Assumption B2}:\ Assume $V^h$ is a subspace of $V_h$.
Let us define the eigenpair approximation $(\lambda^h,u^h)$ by solving the nonlinear eigenvalue problem as follows:

Find $(\lambda^h,u^h)\in\mathcal{R}\times V^h$ such that $b(u^h,u^h)=1$ and
\begin{eqnarray}\label{Nonlinear_Eigenvalue_Problem_subspace}
a(u^h,v^h)&=&\lambda^hb(u^h,v^h),\ \ \ \ \forall v^h\in V^h.
\end{eqnarray}
Then the following error estimates hold
\begin{eqnarray}
\|\bar{u}_h-u^h\|_a &\leq& (1+C_u\eta_a(V^h))\delta_h(\bar{u}_h),\label{Error_Estimate_Eigenfunction_Subspace}\\
|\bar{\lambda}_h-\lambda^h|+\|\bar{u}_h-u^h\|_0&\leq &C_u\eta_a(V^h)\|\bar{u}_h-u^h\|_a,\label{Error_Estimate_Eigenvalue_Subspace}
\end{eqnarray}
where
\begin{eqnarray}\label{Detlat_Definition_Subspace}
\delta_h(\bar{u}_h):=\inf_{v^h\in V^h}\|\bar{u}_h-v^h\|_1.
\end{eqnarray}

\section{Full multigrid algorithm for nonlinear eigenvalue problem}
In this section, a type of full multigrid method is presented. In order to
 describe the full multigrid method, we first introduce
the sequence of finite element spaces. We generate a coarse mesh $\mathcal{T}_H$
with the mesh size $H$ and the coarse linear finite element space $V_H$ is
defined on the mesh $\mathcal{T}_H$. Then a sequence of
 triangulations $\mathcal{T}_{h_k}$
of $\Omega\subset \mathcal{R}^d$ is determined as follows.
Suppose $\mathcal{T}_{h_1}$ (produced from $\mathcal{T}_H$ by
regular refinements) is given and let $\mathcal{T}_{h_k}$ be obtained
from $\mathcal{T}_{h_{k-1}}$ via one regular refinement step
(produce $\beta^d$ subelements) such that
\begin{eqnarray}\label{mesh_size_recur}
h_k=\frac{1}{\beta}h_{k-1},\ \ \ \ k=2,\cdots,n,
\end{eqnarray}
where the positive number $\beta$ denotes the refinement index
 and larger than $1$ (always equals $2$).
Based on this sequence of meshes, the corresponding
 nested linear finite element spaces can be built such that
\begin{eqnarray}\label{FEM_Space_Series}
V_{H}\subseteq V_{h_1}\subset V_{h_2}\subset\cdots\subset V_{h_n}.
\end{eqnarray}
The sequence of finite element spaces
$V_{h_1}\subset V_{h_2}\subset\cdots\subset V_{h_n}$
 and the finite element space $V_H$ have  the following relations
of approximation accuracy (cf. \cite{BrennerScott,Ciarlet})
\begin{eqnarray}\label{delta_recur_relation}
\eta_a(V_H)\geq C\delta_{h_1}(u),\ \ \ \
\delta_{h_k}(u)=\frac{1}{\beta}\delta_{h_{k-1}}(u),\ \ \ k=2,\cdots,n.
\end{eqnarray}

\subsection{One correction step}
In order to design the full multigrid method, we first introduce an one correction
step in this subsection. Assume we have obtained an eigenpair approximation
$(\lambda_{h_k}^{(\ell)},u_{h_k}^{(\ell)})\in \mathcal{R}\times V_{h_k}$, where
$(\ell)$ denote the $\ell$-th iteration step in the $k$-th level finite element
space $V_{h_k}$. In this subsection, a type of correction step to improve
the accuracy of the current eigenpair approximation $(\lambda_{h_k}^{(\ell)},u_{h_k}^{(\ell)})$
will be given as follows.
\begin{algorithm}\label{Multigrid_Smoothing_Step} One Correction Step
\begin{enumerate}
\item Define the following auxiliary boundary value problem:
Find $\widehat{u}_{h_k}^{(\ell+1)}\in V_{h_k}$ such that
\begin{eqnarray}\label{aux_problem}
\widehat{a}(\widehat{u}_{h_k}^{(\ell+1)}, v_{h_k})&=&
(\lambda^{(\ell)}_{h_k}u^{(\ell)}_{h_k}-f(x,u^{(\ell)}_{h_k}),v_{h_k}),\ \
\ \forall v_{h_k}\in V_{h_k}.
\end{eqnarray}
Perform $m$ multigrid iteration steps with the initial value $u_{h_k}^{(\ell)}$
to obtain a new eigenfunction approximation
$\widetilde{u}^{(\ell+1)}_{h_k}\in V_{h_k}$ by
\begin{eqnarray}\label{Multigrid_Iteration_Step}
\widetilde{u}^{(\ell+1)}_{h_k} = MG(V_{h_k},\lambda^{(\ell)}_{h_k}u^{(\ell)}_{h_k}-f(x,u^{(\ell)}_{h_k}),u^{(\ell)}_{h_k},m),
\end{eqnarray}
where $V_{h_k}$ denotes the working space for the multigrid iteration,
$\lambda^{(\ell)}_{h_k}u^{(\ell)}_{h_k}-f(x,u^{(\ell)}_{h_k})$ is the right hand side term of the linear equation,
$u^{(\ell)}_{h_k}$ denotes  the initial guess and $m$ is the number of multigrid iteration times.
\item  Define a new finite element
space $V_{H,h_k}=V_H+{\rm span}\{\widetilde{u}^{(\ell+1)}_{h_k}\}$ and solve
the following eigenvalue problem:
Find $(\lambda^{(\ell+1)}_{h_k},u^{(\ell+1)}_{h_k})\in\mathcal{R}\times V_{H,h_k}$ such
that $b(u^{(\ell+1)}_{h_k},u^{(\ell+1)}_{h_k})=1$ and
\begin{eqnarray}\label{Eigen_Augment_Problem}
a(u^{(\ell+1)}_{h_k},v_{H,h_k})&=&\lambda^{(\ell+1)}_{h_k}b(u^{(\ell+1)}_{h_k},v_{H,h_k}),\ \ \
\forall v_{H,h_k}\in V_{H,h_k}.
\end{eqnarray}
\end{enumerate}
In order to simplify the notation and summarize the above two steps, we define
\begin{eqnarray*}
(\lambda^{(\ell+1)}_{h_k},u^{(\ell+1)}_{h_k})=EigenMG(V_H,\lambda^{(\ell)}_{h_k}, u^{(\ell)}_{h_k},V_{h_k},m).
\end{eqnarray*}
\end{algorithm}
\begin{theorem}\label{Error_Estimate_One_Smoothing_Theorem}
Assume the multigrid iteration
 $\widetilde{u}^{(\ell+1)}_{h_k} = MG(V_{h_k},\lambda^{(\ell)}_{h_k}u^{(\ell)}_{h_k},u^{(\ell)}_{h_k},m)$ of (\ref{aux_problem})
  has the following error reduction rate
\begin{eqnarray}\label{MG_Theta}
\|\widehat{u}^{(\ell+1)}_{h_k}-\widetilde{u}^{(\ell+1)}_{h_k}\|_a
&\leq&\theta \|\widehat{u}^{(\ell+1)}_{h_k}-u^{(\ell)}_{h_k}\|_a,
\end{eqnarray}
and the given eigenpair approximation $(\lambda^{(\ell)}_{h_k},u^{(\ell)}_{h_k})$ has following estimates
\begin{eqnarray}
|\bar{\lambda}_{h_k}-\lambda^{(\ell)}_{h_k}|+\|\bar{u}_{h_k}-u^{(\ell)}_{h_k}\|_0&\leq&
C_u\eta_a(V_H)\|\bar{u}_{h_k}-u^{(\ell)}_{h_k}\|_a.\label{Estimate_h_k_b}
\end{eqnarray}
Under \textbf{Assumptions A} and \textbf{B2}, the resultant eigenpair approximation
$(\lambda^{(\ell+1)}_{h_k},u^{(\ell+1)}_{h_k})\in\mathcal{R}\times V_{h_k}$ produced by
performing Algorithm \ref{Multigrid_Smoothing_Step} has the following error estimates
\begin{eqnarray}
\|\bar{u}_{h_k}-u^{(\ell+1)}_{h_k}\|_a &\leq & \gamma \|\bar{u}_{h_k}-u^{(\ell)}_{h_k}\|_a,\label{Estimate_h_k_1_a}\\
|\bar{\lambda}_{h_k}-\lambda^{(\ell+1)}_{h_k}|+\|\bar{u}_{h_k}-u^{(\ell+1)}_{h_k}\|_0&\leq&
C_u\eta_a(V_H)\|\bar{u}_{h_k}-u^{(\ell+1)}_{h_k}\|_a,\label{Estimate_h_k_1_b}
\end{eqnarray}
where
\begin{eqnarray}\label{Gamma_Definition}
\gamma = \theta + \Big(C_u\theta + \big(1+\theta\big)\big(\widetilde C_u+C_f\big)\big(1+C_u\eta_a(V_H)\big)\Big)\eta_a(V_H)
\end{eqnarray}
and $\widetilde C_u$ depends on the desired eigenpair. 
\end{theorem}
\begin{proof}
From (\ref{fem_nonlinear_pde}), (\ref{Err_functional_Global_0_Norm}) and (\ref{aux_problem}), we have
\begin{eqnarray*}\label{One_Correction_1}
\widehat{a}(\bar{u}_{h_k}-\widehat{u}^{(\ell+1)}_{h_k}, v_{h_k}\big)
&=&\Big(\big(\bar{\lambda}_{h_k}\bar{u}_{h_k}-\lambda^{(\ell)}_{h_k}u^{(\ell)}_{h_k}\big)
+\big(f(x,\bar{u}_{h_k})-f(x,u^{(\ell)}_{h_k})\big),v_{h_k}\Big),\nonumber\\
&\leq& |\bar\lambda_{h_k}|\|\bar u_{h_k}-u_{h_k}^{(\ell)}\|_0\|v_{h_k}\|_0+|\bar\lambda_{h_k}-\lambda_{h_k}^{(\ell)}|\|u_{h_k}^{(\ell)}\|\nonumber\\
&&\ \ \ \ +C_f\|\bar u_{h_k}-u_{h_k}^{(\ell)}\|_0\|v_{h_k}\|_a,\ \ \  \forall v_{h_k}\in V_{h_k}.
\end{eqnarray*}
It leads to the following estimates by using the property of $\widehat{a}(\cdot,\cdot)$ 
 and (\ref{Estimate_h_k_b})
\begin{eqnarray}\label{One_Correction_2}
\|\bar{u}_{h_k}-\widehat{u}^{(\ell+1)}_{h_k}\|_a &\leq&
 (\widetilde C_u+C_f)\eta_a(V_H)\|\bar{u}_{h_k}-u^{(\ell)}_{h_k}\|_a,
\end{eqnarray}
where $\widetilde C_u$ depends on the desired eigenpair.  

Combining (\ref{MG_Theta}) and (\ref{One_Correction_2}) leads to the following
error estimate for $\widetilde{u}^{(\ell+1)}_{h_k}$
\begin{eqnarray}\label{One_Correction_3}
\|\widehat{u}^{(\ell+1)}_{h_k}-\widetilde{u}^{(\ell+1)}_{h_k}\|_a &\leq& \theta\|\widehat{u}^{(\ell+1)}_{h_k}-u^{(\ell)}_{h_k}\|_a\nonumber\\
&\leq& \theta\big(\|\widehat{u}^{(\ell+1)}_{h_k}-\bar{u}_{h_k}\|_a
 +\|\bar{u}_{h_k}-u^{(\ell)}_{h_k}\|_a\big)\nonumber\\
&\leq& \theta\big(1+(\widetilde C_u+C_f)\eta_a(V_H)\big)\|\bar{u}_{h_k}-u^{(\ell)}_{h_k}\|_a.
\end{eqnarray}
Then from (\ref{One_Correction_2}) and (\ref{One_Correction_3}), we have the following inequalities
\begin{eqnarray}\label{One_Correction_4}
\|\bar{u}_{h_k}-\widetilde{u}^{(\ell+1)}_{h_k}\|_a&\leq& \|\bar{u}_{h_k}-\widehat{u}^{(\ell+1)}_{h_k}\|_a
+\|\widehat{u}^{(\ell+1)}_{h_k}-\widetilde{u}^{(\ell+1)}_{h_k}\|_a\nonumber\\
&\leq& \big(\theta+(1+\theta)(\widetilde C_u+C_f)\eta_a(V_H)\big)\|\bar{u}_{h_k}-u^{(\ell)}_{h_k}\|_a.
\end{eqnarray}

The eigenvalue problem (\ref{Eigen_Augment_Problem})
can be regarded  as a finite dimensional subspace approximation of the eigenvalue problem (\ref{fem_nonlinear_pde}).
Using (\ref{Error_Estimate_Eigenfunction_Subspace}) and (\ref{Error_Estimate_Eigenvalue_Subspace}) in \textbf{Assumption B2},
the following estimates hold
\begin{eqnarray}\label{Error_u_u_h_2}
\|\bar{u}_{h_k}-u^{(\ell+1)}_{h_k}\|_a &\leq& \big(1+C_u\eta_a(V_{H,h_k})\big)\inf_{v_{H,h_k}\in
V_{H,h_k}}\|\bar{u}_{h_k}-v_{H,h_k}\|_a\nonumber\\
&\leq& \big(1+C_u\eta_a(V_H)\big)\|\bar{u}_{h_k}-\widetilde{u}_{h_k}^{(\ell+1)}\|_a\nonumber\\
&\leq& \gamma \|\bar{u}_{h_k}-u_{h_k}^{(\ell)}\|_a,
\end{eqnarray}
and
\begin{eqnarray}\label{Error_u_u_h_2_Negative}
|\bar{\lambda}_{h_k}-\lambda_{h_k}^{(\ell+1)}|+\|\bar{u}_{h_k}-u_{h_k}^{(\ell+1)}\|_0&\leq&
C_u\eta_a(V_{H,h_k})\|\bar{u}_{h_k}-u_{h_k}^{(\ell+1)}\|_a\nonumber\\
&\leq& C_u\eta_a(V_H)\|\bar{u}_{h_k}-u_{h_k}^{(\ell+1)}\|_a.
\end{eqnarray}
Then we obtained the desired results (\ref{Estimate_h_k_1_a}) and (\ref{Estimate_h_k_1_b})
and the proof is complete.
\end{proof}

\subsection{Full multigrid method for eigenvalue problem}
In this subsection, based on the one correction step defined in Algorithm
\ref{Multigrid_Smoothing_Step}, a type of full multigrid scheme
will be introduced. The optimal error estimate with the optimal computational work
will be deduced for this type of full multigrid method.

Since the multigrid method for the boundary value problem has the uniform
error reduction rate (cf. \cite{BrennerScott,Hackbusch_Book}), we can choose suitable $m$
such that $\theta<1$ in (\ref{MG_Theta}). From the definition (\ref{Gamma_Definition}) for $\gamma$,
it is obvious that $\gamma<1$ when the mesh size $H$ of $\mathcal{T}_H$ is small enough.
Based on these property, we can design a full multigrid method for nonlinear eigenvalue problem as follows.
\begin{algorithm}\label{Full_Multigrid}Full Multigrid Scheme
\begin{enumerate}
\item Solve the following nonlinear eigenvalue problem in $V_{h_1}$:
Find $(\lambda_{h_1}, u_{h_1})\in \mathcal{R}\times V_{h_1}$ such that $b(u_{h_1},u_{h_1})=1$ and
\begin{equation*}
a(u_{h_1}, v_{h_1}) = \lambda_{h_1} b(u_{h_1}, v_{h_1}), \quad \forall v_{h_1}\in  V_{h_1}.
\end{equation*}
Solve this nonlinear eigenvalue problem to get the desired eigenpair approximation
$(\lambda_{h_1},u_{h_1})\in\mathcal{R}\times V_{h_1}$.
\item For $k=2,\cdots,n$, do the following iterations
\begin{itemize}
\item Set $\lambda_{h_k}^{(0)}=\lambda_{h_{k-1}}$ and $u_{h_k}^{(0)} = u_{h_{k-1}}$.
\item Perform the following multigrid iterations
\begin{eqnarray*}
(\lambda^{(\ell+1)}_{h_k}, u^{(\ell+1)}_{h_k})=
 EigenMG(V_H,\lambda^{(\ell)}_{h_k},u^{(\ell)}_{h_k},V_{h_k},m),
\ \ \   {\rm for}\ \ell=0,\cdots, p-1.
\end{eqnarray*}
\item Set $\lambda_{h_k}=\lambda^{(p)}_{h_k}$ and $u_{h_k}=u^{(p)}_{h_k}$.
\end{itemize}
End Do
\end{enumerate}
Finally, we obtain an eigenpair approximation
$(\lambda_{h_n},u_{h_n})\in \mathcal{R}\times V_{h_n}$ in the finest space.
\end{algorithm}
\begin{theorem}\label{Error_Full_Multigrid_Theorem}
Assume the conditions of Theorem \ref{Error_Estimate_One_Smoothing_Theorem}
and \textbf{Assumption B1} hold.
After implementing Algorithm \ref{Full_Multigrid}, the resultant
eigenpair approximation $(\lambda_{h_n},u_{h_n})$ has the following
error estimate
\begin{eqnarray}
\|\bar{u}_{h_n}-u_{h_n}\|_a &\leq&
C\frac{\gamma^p}{1-\beta\gamma^p}\delta_{h_n}(u),\label{FM_Err_fun}\\
|\bar{\lambda}_{h_n}-\lambda_{h_n}| + \|\bar{u}_{h_n}-u_{h_n}\|_0 &\leq&
 C\frac{\gamma^p}{1-\beta\gamma^p}\eta_a(V_H)\delta_{h_n}(u),\label{FM_Err_eigen}
\end{eqnarray}
under the condition $\beta\gamma^p<1$.
\end{theorem}
\begin{proof}
Define $e_{k}:=\bar{u}_{h_k}-u_{h_k}$. Then from step 1
in Algorithm \ref{Full_Multigrid}, it is obvious $e_1=0$.
For $k=2,\cdots,n$, from \textbf{Assumption B1} and
Theorem \ref{Error_Estimate_One_Smoothing_Theorem},  we have
\begin{eqnarray}\label{FM_Estimate_1}
\|e_k\|_a&\leq& \gamma^p\|\bar{u}_{h_k}-u_{h_{k-1}}\|_a\nonumber\\
&\leq& \gamma^p\big(\|\bar{u}_{h_k}-\bar{u}_{h_{k-1}}\|_a
+\|\bar{u}_{h_{k-1}}-u_{h_{k-1}}\|_a\big)\nonumber\\
&\leq& \gamma^p \big(C\delta_{h_k}(u)+\|e_{k-1}\|_a\big).
\end{eqnarray}
By iterating inequality (\ref{FM_Estimate_1}) and the condition $\beta\gamma^p<1$, the following
inequalities hold
\begin{eqnarray}
\|e_n\|_a&\leq& C\gamma^p\delta_{h_n}(u)+C\gamma^{2p}\delta_{h_{n-1}}(u)+\cdots
+C\gamma^{(n-1)p}\delta_{h_2}(u)\nonumber\\
&\leq& C\sum_{k=2}^n \gamma^{(n-k+1)p}\delta_{h_{\ell}}(u)
=C\left(\sum_{k=2}^n \big(\beta\gamma^{p}\big)^{n-k}\right)\gamma^{p}\delta_{h_n}(u)\nonumber\\
&\leq&C\frac{\gamma^p}{1-\beta\gamma^p}\delta_{h_n}(u).
\end{eqnarray}
For such choice of $p$, we arrive the desired result (\ref{FM_Err_fun}) and (\ref{FM_Err_eigen}) can be obtained by
(\ref{Error_Estimate_Eigenvalue_Subspace}), (\ref{Estimate_h_k_1_b}) and (\ref{FM_Err_fun}).
\end{proof}
\begin{remark}
The good convergence rate of the multigrid method for boundary value problems leads to that
we do not need to choose large $m$ and $p$ \cite{BrennerScott,Hackbusch_Book,Shaidurov,Xu}.
\end{remark}
\subsection{Estimate of the computational work}
In this subsection, we turn our attention to the estimate of computational work
for the full multigrid method defined in Algorithm \ref{Full_Multigrid}. It will be shown that
the full multigrid method makes solving the nonlinear eigenvalue problem need almost the same work as
solving the corresponding linear boundary value problems.

First, we define the dimension of each level
finite element space as $N_k:={\rm dim}V_{h_k}$. Then we have
\begin{eqnarray}\label{relation_dimension}
N_k\approx\Big(\frac{1}{\beta}\Big)^{d(n-k)}N_n,\ \ \ k=1,2,\cdots, n.
\end{eqnarray}

The computational work for the second step in Algorithm \ref{Multigrid_Smoothing_Step} is different
from the linear eigenvalue problems \cite{LinXie,Xie_IMA,Xie_Nonconforming,Xie_JCP}.
 In this step, we need to solve a nonlinear eigenvalue problem (\ref{Eigen_Augment_Problem}).
Always, some type of nonlinear iteration method (self-consistent iteration or
 Newton type iteration) is adopted to solve this nonlinear eigenvalue problem.
In each nonlinear iteration step, it is required to assemble the matrix on the finite element
space $V_{H,h_k}$ ($k=2,\cdots,n$) which needs the computational work $\mathcal{O}(N_k)$.
Fortunately, the matrix assembling can be carried out by the parallel way easily
in the finite element space since it has no data transfer.

\begin{theorem}\label{optimal_work}
Assume we use $\vartheta$ computing-nodes in Algorithm \ref{Full_Multigrid}, the nonlinear
eigenvalue solving in the coarse spaces $V_{H,h_k}$ ($k=1,\cdots, n$) and
$V_{h_1}$ need work $\mathcal{O}(M_H)$ and $\mathcal{O}(M_{h_1})$, respectively,
and the work of the multigrid solver
$MG(V_{h_k},\lambda^{(\ell)}_{h_k}u^{(\ell)}_{h_k}$, 
$u^{(\ell)}_{h_k},m)$
in each level space $V_{h_k}$ is $\mathcal{O}(N_k)$ for $k=2,3,\cdots,n$.
Let $\varpi$ denote the nonlinear iteration times when we solve
the nonlinear eigenvalue problem (\ref{Eigen_Augment_Problem}).
Then in each computational node, the work involved
in Algorithm \ref{Full_Multigrid} has the following estimate
\begin{eqnarray}\label{Computation_Work_Estimate}
{\rm Total\ work}&=&\mathcal{O}\Big(\big(1+\frac{\varpi}{\vartheta}\big)N_n
+ M_H\log N_n+M_{h_1}\Big).
\end{eqnarray}
\end{theorem}
\begin{proof}
We use $W_{k}$ to denote the work involved in each correction step on the $k$-th
finite element space $V_{h_{k}}$.
From the definition of Algorithm \ref{Multigrid_Smoothing_Step}, we have the following estimate
\begin{eqnarray}\label{work_k}
W_k&=&\mathcal{O}\left(N_k +M_H+\varpi\frac{N_k}{\vartheta}\right).
\end{eqnarray}
Based on the property (\ref{relation_dimension}), iterating (\ref{work_k}) leads to
\begin{eqnarray}\label{Work_Estimate}
\text{Total work} &=& \sum_{k=1}^nW_k\nonumber =
\mathcal{O}\left(M_{h_1}+\sum_{k=2}^n
\Big(N_k + M_H+\varpi\frac{N_k}{\vartheta}\Big)\right)\nonumber\\
&=& \mathcal{O}\left(\sum_{k=2}^n\Big(1+\frac{\varpi}{\vartheta}\Big)N_k
+ (n-1) M_H + M_{h_1}\right)\nonumber\\
&=& \mathcal{O}\left(\sum_{k=2}^n
\Big(\frac{1}{\beta}\Big)^{d(n-k)}\Big(1+\frac{\varpi}{\vartheta}\Big)N_n
+ M_H\log N_n+M_{h_1}\right)\nonumber\\
&=& \mathcal{O}\left(\big(1+\frac{\varpi}{\vartheta}\big)N_n
+ M_H\log N_n+M_{h_1}\right).
\end{eqnarray}
This is the desired result and we complete the proof.
\end{proof}
\begin{remark}
Since we have a good enough initial solution $\widetilde{u}_{h_{k+1}}$
in the second step of Algorithm \ref{Multigrid_Smoothing_Step},
then solving the nonlinear eigenvalue problem (\ref{Eigen_Augment_Problem}) always does not
need many nonlinear iteration  times (always $\varpi\leq 3$).
In this case, the complexity in each computational node will be $\mathcal{O}(N_n)$ provided
$M_H\ll N_n$ and $M_{h_1}\leq N_n$.
\end{remark}
\section{Numerical results}
In this section, two numerical examples are presented to illustrate the
efficiency of the full multigrid scheme proposed in this paper.

\begin{example}\label{Example_1}
In this example, we consider the ground state solution of Gross-Pitaevskii equation (GPE) for
Bose-Einstein condensation (BEC).
\begin{equation}\label{nonlinear_pde}
\left\{
\begin{array}{rcl}
-\triangle u+Wu +\zeta|u|^2u&=&\lambda u, \quad {\rm in} \  \Omega,\\
u&=&0, \ \  \quad {\rm on}\  \partial\Omega,\\
\int_{\Omega}u^2d\Omega&=&1,
\end{array}
\right.
\end{equation}
where $\Omega$ denotes the three dimensional domain $[0,1]^3$ , $ \zeta=1$
and $W=x_1^2+x_2^2+x_3^2$.
\end{example}
From the results \cite{CancesChakirMaday,XieXie}, the \textbf{Assumptions A}, \textbf{B1} and \textbf{B2} hold
for the GPE (\ref{nonlinear_pde}). So the proposed full multigrid method can be applied
to the GPE (\ref{nonlinear_pde}).
\begin{figure}
\centering
\includegraphics[width=8 cm]{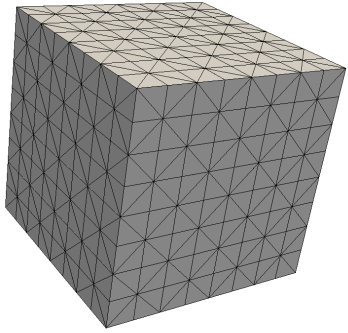}
\caption{The initial mesh for Examples \ref{Example_1}}\label{mesh}
\end{figure}
The sequence of finite elements spaces are constructed by linear element on a series of meshes
produced by regular refinement with $\beta=2$. In each level of the full multigrid scheme
 defined in Algorithm \ref{Full_Multigrid}, the parameters are set to be
$m=1$, $p=1$. And we take 3 conjugate gradient smooth steps for the presmoothing
and postsmoothing  iteration step in the multigrid iteration in the step 1 of
Algorithm \ref{Multigrid_Smoothing_Step}. Since the exact solution is not known,
an adequate accurate approximation is chosen as the exact solution for our numerical test.
Figure \ref{mesh} shows the corresponding initial mesh.

Figure \ref{error of bec1} gives the corresponding numerical results of Algorithm \ref{Full_Multigrid}. From Figure
\ref{error of bec1}, we can find that the full multigrid
scheme can
obtain the optimal error estimates for both eigenvalue and eigenvector.

In order to show the efficiency of Algorithm \ref{Full_Multigrid},
we provide the CPU time for Algorithm \ref{Full_Multigrid}.
Here, we choose the Package ARPACK as the eigenvalue solving tool and the
full multigrid scheme is running on the
machine PowerEdge R720 with the linux system. The corresponding results
are presented in Table \ref{table1} which shows the efficiency and linear
complexity of Algorithm \ref{Full_Multigrid}.

\begin{table}[ht]
\centering
\caption{The CPU time for Example \ref{Example_1} by Algorithm \ref{Full_Multigrid}}\label{table1}
\begin{tabular}{|c|c|c|}\hline
Number of levels  & Number of elements & Time for Algorithm \ref{Full_Multigrid}\\
\hline
1 & 3072          & 0.45           \\
\hline
2 & 24576         & 1.55           \\
\hline
3 & 196608        & 8.08           \\
\hline
4 & 1572846       & 63.01          \\
\hline
5 & 12582912      & 519.86         \\
\hline
\end{tabular}
\end{table}

\begin{figure}[ht]
\centering
\includegraphics[width=7cm]{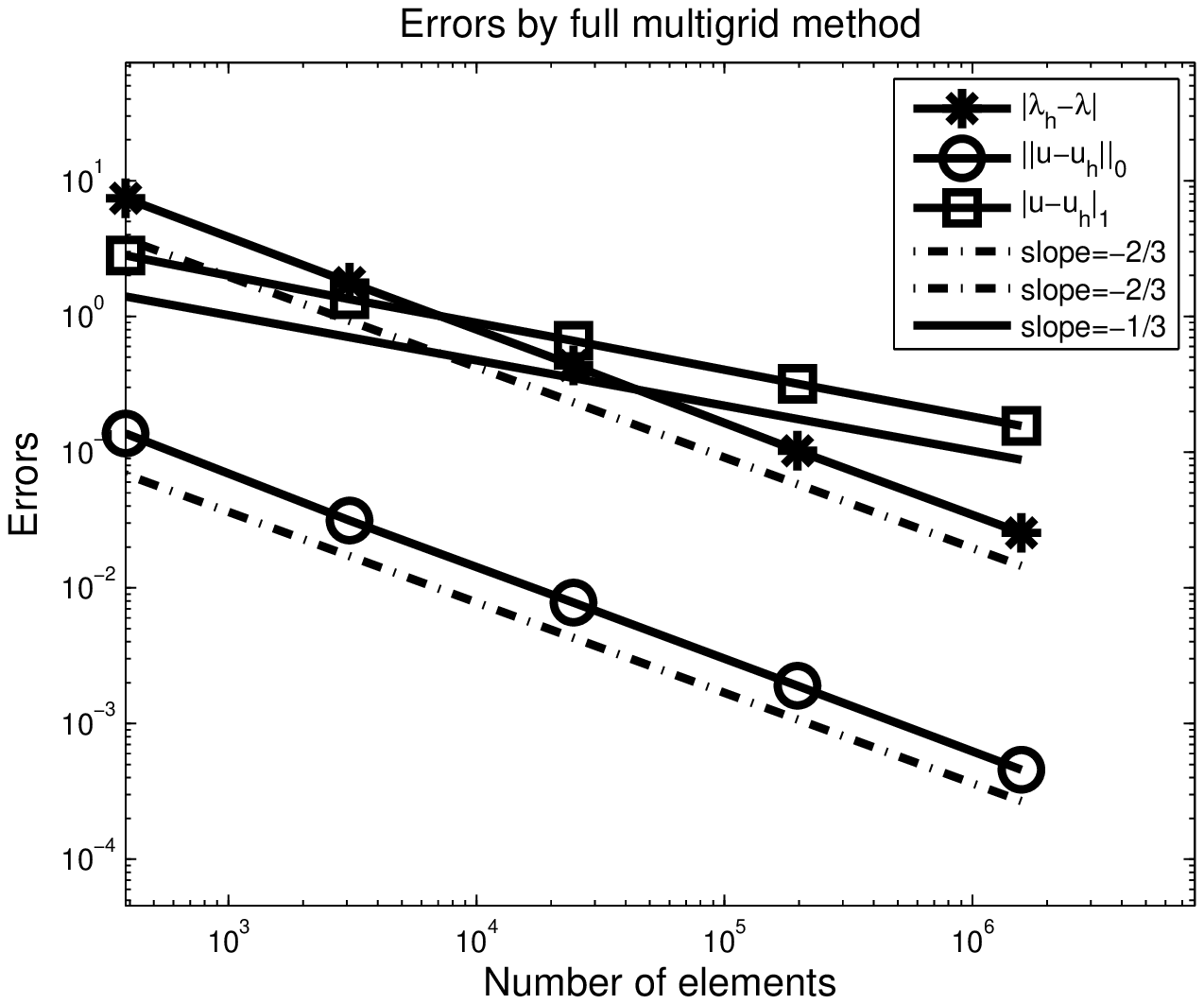}
\includegraphics[width=7cm]{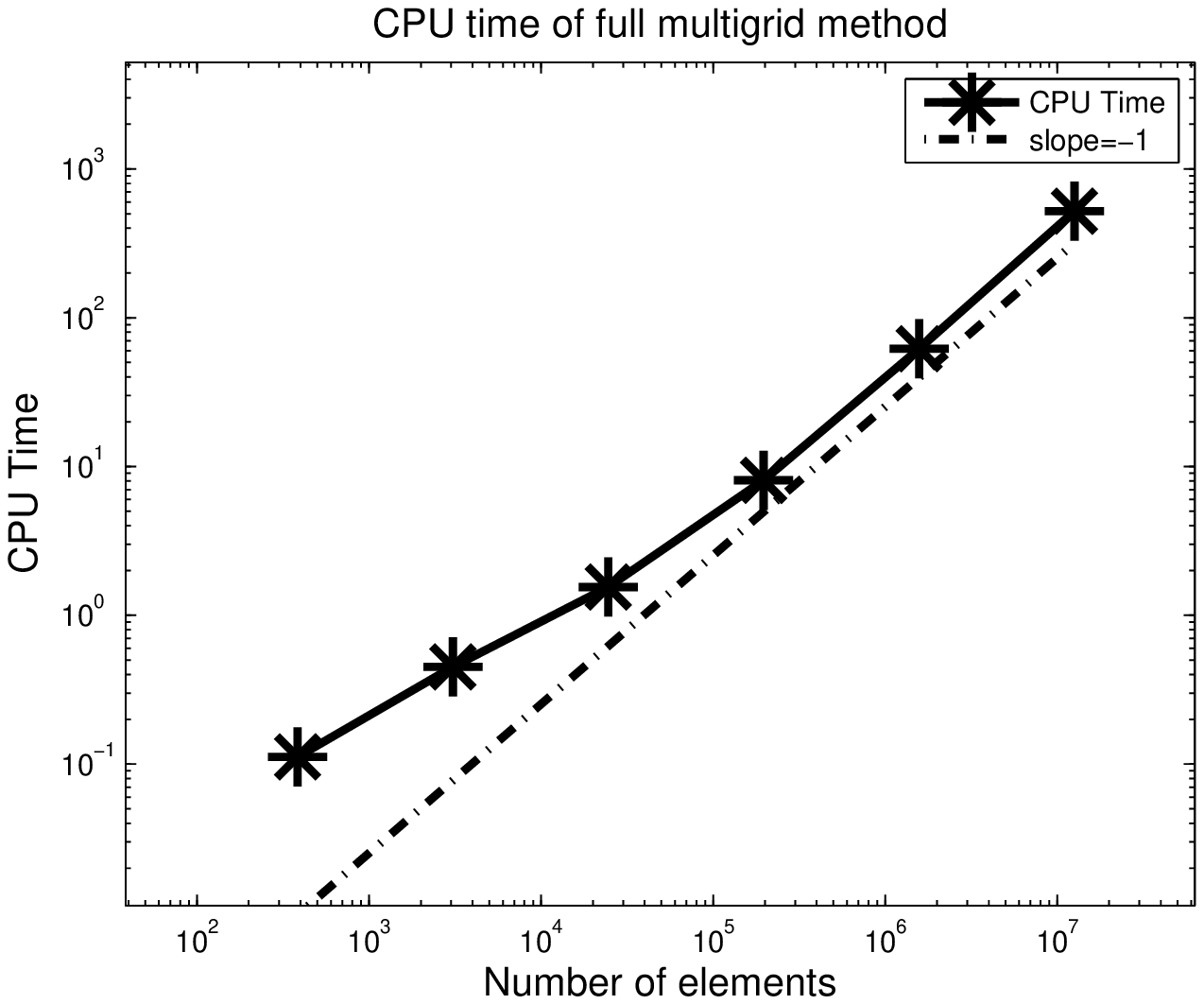}
\caption{Left: The errors of the full multigrid method for the ground state solution of GPE,
where $\lambda_h$ and $u_h$ denote the numerical eigenvalue and eigenfunction
by Algorithm \ref{Full_Multigrid}. Right: CPU Time of Algorithm \ref{Full_Multigrid}
for Example \ref{Example_1}}\label{error of bec1}
\end{figure}

\begin{example}\label{Example_2}
In the second example, we consider the GPE with the coefficient $\zeta=100$ and $W=x_1^2+x_2^2+x_3^2$
on the domain $\Omega = [0,1]^3$.
\end{example}
The initial mesh used in this example is the one shown in Figure \ref{mesh}.
Numerical results are present in Table \ref{table2} and Figure \ref{error of bec100}.
It is obvious, Table \ref{table2} and Figure \ref{error of bec100} also show that
the efficiency and linear complexity of  Algorithm \ref{Full_Multigrid}.

\begin{table}[ht]
\centering
\caption{The CPU time for Example \ref{Example_2} by Algorithm \ref{Full_Multigrid}}\label{table2}
\begin{tabular}{|c|c|c|}\hline
Number of levels  & Number of elements & time for Algorithm \ref{Full_Multigrid} \\
\hline
1 & 24576         & 4.32             \\
\hline
2 & 196608        & 11.43            \\
\hline
3 & 1572846       & 70.88            \\
\hline
4 & 12582912      & 577.52           \\
\hline
\end{tabular}
\end{table}

\begin{figure}[ht]
\centering
\includegraphics[width=7cm]{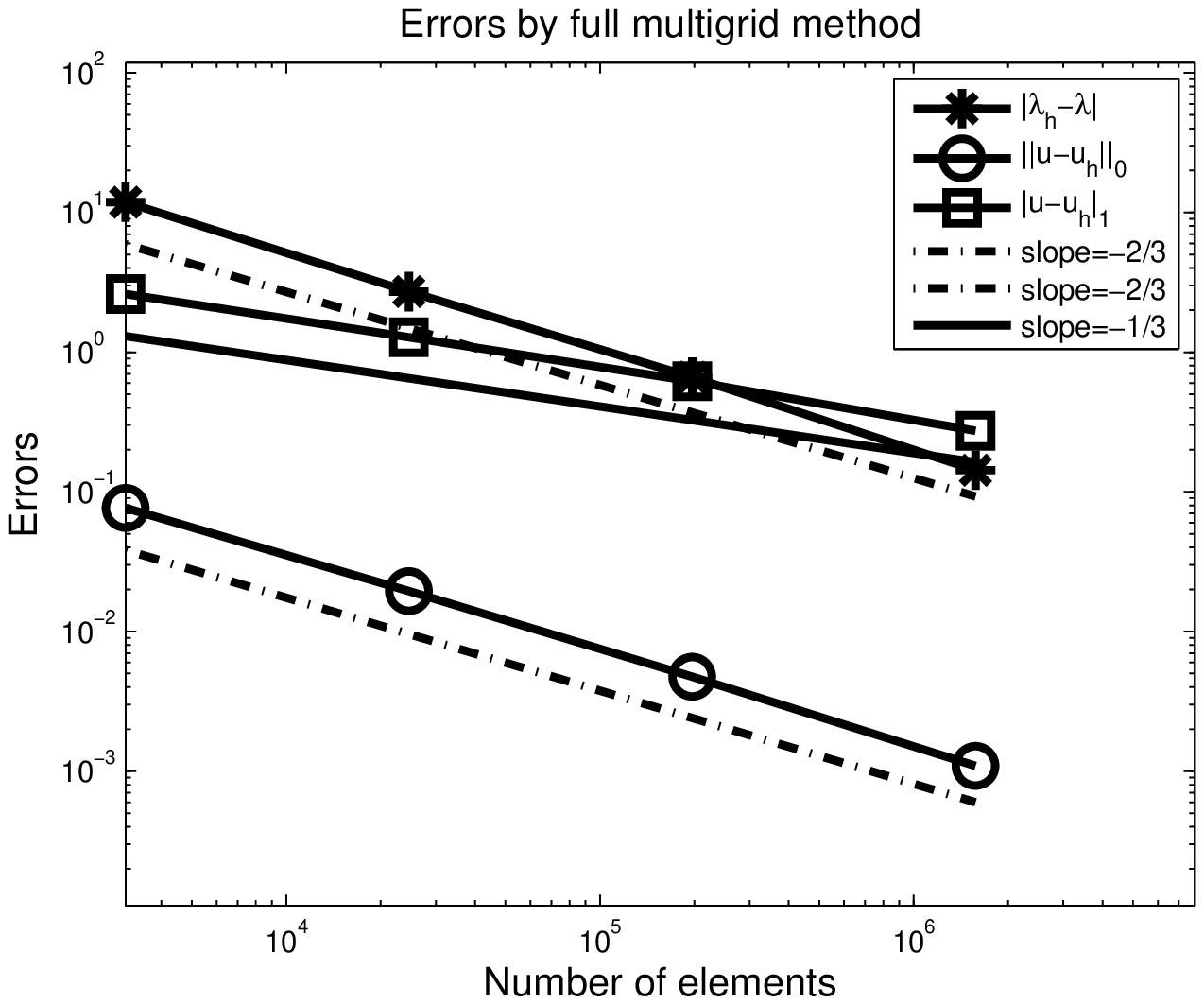}
\includegraphics[width=7cm]{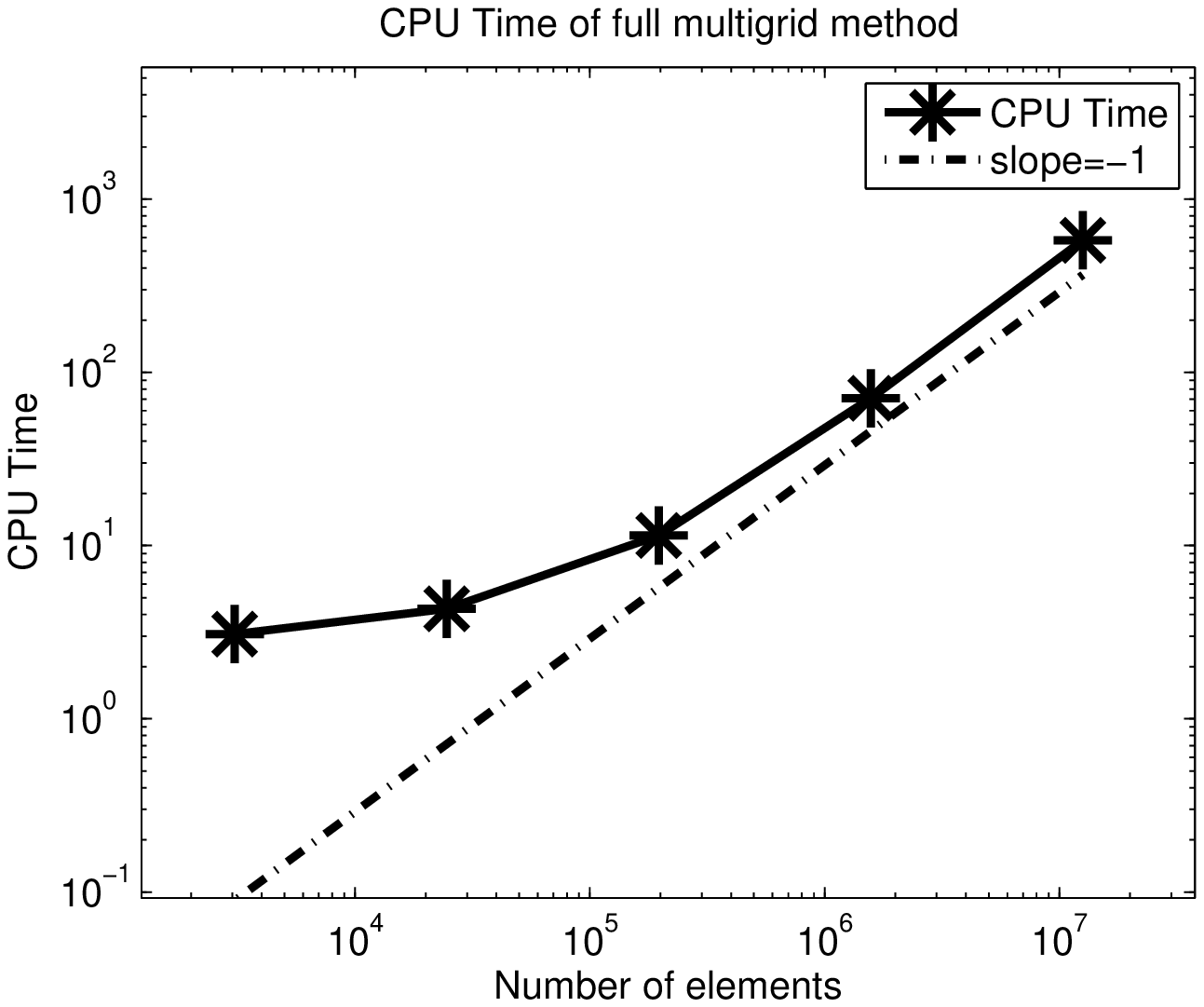}
\caption{Left: The errors of the full multigrid method for the ground state solution of GPE,
where $\lambda_h$ and $u_h$ denote the numerical eigenvalue and eigenfunction
by Algorithm \ref{Full_Multigrid}. Right: CPU Time of Algorithm \ref{Full_Multigrid}
for Example \ref{Example_2}}\label{error of bec100}
\end{figure}

\section{Concluding remarks}
In this paper, a type of full multigrid method is introduced for nonlinear eigenvalue problems.
The proposed methods is based on the combination of the multilevel correction technique for
nonlinear eigenvalue problems and the multigrid iteration for linear boundary value problems.
The multilevel correction technique can transform the nonlinear eigenvalue solving into
a series of solutions of linear boundary value problems on a sequences of finite element spaces.
The multigrid iteration is one of the efficient iteration which has uniform error reduction rate.

The multigrid iteration can also be replaced by other types of efficient iteration schemes
such as algebraic multigrid method, the type of preconditioned schemes based on
the subspace decomposition and subspace corrections (see, e.g., \cite{BrennerScott, Xu}) and the
domain decomposition method (see, e.g., \cite{ToselliWidlund,XuZhou_Eigen_LocalParallel}).


\end{document}